\documentclass[12pt]{article}
\usepackage{amsmath, amsthm, amssymb,enumitem,multirow,rotating,longtable,slashbox,hhline}

\newtheorem{prop}{Proposition}[section]
\newtheorem{thm}[prop]{Theorem}
\newtheorem{lemma}[prop]{Lemma}
\newtheorem{cor}[prop]{Corollary}
\begin{document}
\begin{center}
{\large \bf Some Notes on Pairs in Binary Strings}\\
Jeremy M. Dover
\end{center}
\begin{abstract}
Seth~\cite{1812699} posed a problem that is equivalent to the following: how many binary strings of length $n$ have exactly $k$ pairs of consecutive 0s and exactly
$m$ pairs of consecutive 1s, where the first and last bits are considered as being consecutive? In this paper, we provide a closed form solution which also solves a
related problem with some interesting connections to other combinatorial sequences.
\end{abstract}

\section{The Setup}
Seth~\cite{1812699} posed the following problem: consider a microstate consisting of 8 spins, where a microstate is a linear ordering of spins, each of which may be in the up or down state. Seth asks how many of these 8-spin microstates have exactly 2 ``up parallel pairs" and 2 ``down parallel pairs", where an ``up parallel pair" is two consecutive up states, and the obvious meaning of consecutive for the linear ordering is extended so that the first and last states are also considered consecutive. A ``down parallel pair" is defined analogously. Note that despite the first and last states being considered consecutive, the microstate is still considered to have a first and last state, so rotations of the state pattern are counted as being different.

It is not hard to cast this problem into a question about binary strings, where an ``up spin" is a 0, and a ``down spin" a 1, namely finding the cardinality of the set $S^\circ(n,k,m)$, the set of all binary strings of length $n$ with $k$ pairs of consecutive 0s and $m$ pairs of consecutive 1s, with the first and last bits considered consecutive. In order to address this problem, we define the related set $Z(n,k,m)$ to be the number of binary sequences of length $n$ that start with 0 and have $k$ pairs of consecutive zeroes, and $m$ pairs of consecutive ones, where the first and last bits are {\em not} considered consecutive. We denote the cardinality of $Z(n,k,m)$ as $z(n,k,m)$. We now show that $\left|S^\circ(n,k,m)\right|$ can be determined in terms of the values $z(n,k,m)$. In what follows, we will assume unless otherwise stated that the first and last bits of a binary string are not considered consecutive. For brevity, we refer to a pair of consecutive 0s (resp. 1s) in a binary string as a 0-pair (resp. 1-pair); this terminology will specifically not be used for the first and last bits, in those cases where they are considered consecutive.

The first issue to address is that some of the binary strings in $S^\circ(n,k,m)$ begin with 1. However, we note that the operation of inverting each bit of a binary string of length $n$ is obviously a bijective involution from the set of all binary strings of length $n$ onto itself, and shows that the number of binary strings that start with a 1 and have $k$ 0-pairs and $m$ 1-pairs is exactly $z(n,m,k)$.

We know that the elements of $Z(n,k,m)$ begin with 0, but we do not necessarily know how they end, which is an important consideration when analyzing $S^\circ(n,k,m)$. The following lemma provides an answer to this question.

\begin{lemma}
\label{end}
Let $n,k,m$ be integers such that $n \ge 1 $ and $k,m \ge 0$, and let $b$ be a binary string of length $n$ containing $k$ 0-pairs and $m$ 1-pairs. Then the last bit of $b$ is the same as the first bit of $b$ if and only if $n+k+m$ is odd.
\end{lemma}

\begin{proof}
Given a binary string $b$ of length $n$, assign to each pair of consecutive bits (of which there are $n-1$) the letter S if they are the same, and D if they are different. Since $b$ has $k$ 0-pairs and $m$ 1-pairs, there are exactly $m+k$ Ss, and thus there are $n-1-m-k$ Ds. Reading from left to right, we only change values in the string when we encounter a D, so it is easy to see that the last bit of the string depends only on the parity of $n-1-m-k$. If this value is odd, then the last bit of $b$ is different from the first bit, while these bits will be the same if the number of Ds is even. Since $n+m+k$ and $n-1-m-k$ have opposite parity, we obtain the result.
\end{proof}

Let's use these facts to determine $\left|S^\circ(n,k,m)\right|$. Let $b \in S^\circ(n,k,m)$ be a binary string. If the first and last bits of $b$ are different, then by Lemma~\ref{end} we must have $n+k+m$ even, since $b$ has $k$ 0-pairs and $m$ 1-pairs. If the first and last bits of $b$ are the same, then $b$ has either $k-1$ 0-pairs and $m$ 1-pairs, or $k$ 0-pairs and $m-1$ 1-pairs; in either case Lemma~\ref{end} shows that $n+k+m-1$ must be odd, or $n+k+m$ is even. This shows that if $n+k+m$ is odd, then there are {\em no} binary strings in $S^\circ(n,k,m)$. But if $n+k+m$ is even, then all of the binary strings in $Z(n,k,m)$ and $Z(n,k-1,m)$ are in $S^\circ(n,k,m)$, as are the inversed of the string in $Z(n,m,k)$ and $Z(n,m-1,k)$. Thus we have:
$$S^\circ(n,k,m) = \begin{cases} 0 &\mbox n+k+m\,{\rm odd}\\z(n,k,m)+z(n,k-1,m)+\\z(n,m,k)+z(n,m-1,k) & \mbox n+k+m\,{\rm even}\end{cases}$$

The important takeaway from this section is that our original problem can be solved strictly by consideration of the numbers $z(n,k,m)$, which we focus on exclusively in what follows.

\section{Some recurrence relations for $z(n,k,m)$}

The boundary conditions for $z(n,k,m)$ are fairly straightforward; for convenience we define $z(n,k,m) = 0$ for all $n < 0$, $k < 0$ or $m < 0$. Noting that a binary string of length $n$ only has $n-1$ pairs of consecutive bits, we know that $z(n,k,m) = 0$ for all integers $k,m$ such that $k+m \ge n$. Moreover, if $k+m = n-1$, our numbers $z(n,k,m)$ count the number of strings where all pairs of consecutive bits are identical, with the first bit of the string being 0. This forces the string to be entirely 0s, showing that for integers $k,m$ with $k+m = n-1$, $z(n,k,m) = 0$ unless $k=n-1$ and $m=0$, in which case $z(n,n-1,0) = 1$.

We now derive several different recurrence relations for the $z(n,k,m)$, which each have different uses.

\begin{thm}
\label{recur1}
Let $n,k,m$ be integers such that $n \ge 3$ and $k,m \ge 0$. Then $z(n,k,m) = z(n-1,k-1,m) + z(n-2,k,m) + z(n-2,m-1,k)$.
\end{thm}

\begin{proof}
To prove this result, we count the size of $Z(n,k,m)$ in two ways, one of which is $z(n,k,m)$ by definition. For the other count, let $b \in Z(n,k,m)$ and consider three cases:
\begin{enumerate}
\item If $b$ starts with 00, then $b$ consists of a 0 followed by a string of $n-1$ bits starting with 0 with $k-1$ 0-pairs and $m$ 1-pairs, of which there are $z(n-1,k-1,m)$.
\item If $b$ starts with 010, then $b$ consists of 01 followed by a string of $n-2$ bits starting with 0 with $k$ 0-pairs and $m$ 1-pairs, of which there are $z(n-2,k,m)$.
\item If $b$ starts with 011, then $b$ consists of 01 followed by a string of $n-2$ bits starting with 1 with $k$ 0-pairs and $m-1$ 1-pairs, of which there are $z(n-2,m-1,k)$.
\end{enumerate}
Note that if $m=0$, no strings in $Z(n,k,m)$ start with 011, but $z(n-2,-1,k)$ is defined to be 0, so this remains correct. Since these three cases count sets that form a disjoint union of $Z(n,k,m)$, we have the result.
\end{proof}

Our next recurrence is somewhat more complicated and is not universally applicable, but it counts in a way which quickly reveals an important corollary.

\begin{thm}
\label{recur2}
Let $n,k,m$ be integers such that $n \ge 1$ and $k+m < n-1$. Then $$z(n,k,m) = \displaystyle \sum_{f=1}^{k+1} z(n-f,m,k+1-f)$$
\end{thm}

\begin{proof}
Again we proceed by counting the cardinality of $Z(n,k,m)$ in two ways, the first yielding $z(n,k,m)$. To count this set in another way, we note that since $k+m < n-1$, any $b \in Z(n,k,m)$ must contain at least one 1. Let $f$ be the position of the first 1 in $b$, so $f$ may vary between $1$ and $n-1$ (note: the first bit of $b$ has index 0). Prior to the first 1, the string consists entirely of 0s, creating $f-1$ pairs of consecutive 0s. Starting from the first 1, the remainder of the string is a binary string of length $n-f$ starting with 1 that contains exactly $k+1-f$ 0-pairs and $m$ 1-pairs, of which there are $z(n-f,m,k+1-f)$ such strings. Thus we can write $$z(n,k,m) = \displaystyle \sum_{f=1}^{n-1} z(n-f,m,k+1-f)$$ Noting that $z(n-f,m,k+1-f) = 0$ for all $f > {\rm min}\{n,k+1\}$ and that $k+1 < n$ yields the indices of summation shown in the Theorem statement.
\end{proof}

\begin{cor}
\label{zeroflip}
Let $n,m$ be integers with $n \ge 1$ and $0 \le m < n-1$. Then $z(n,0,m) = z(n-1,m,0)$.
\end{cor}

\begin{proof}
Since $m < n-1$, we can apply Theorem~\ref{recur2} to $z(n,0,m)$ to obtain $z(n,0,m) = \sum_{f=1}^{1} z(n-f,m,1-f)$. Evaluating the single term of the summation yields the result.
\end{proof}

Our final recurrence is nice because it shows how we can reduce the problem of calculating $z(n,k,m)$ to just those values with $m=0$.

\begin{thm}
\label{generalm}
Let $n,k,m$ be integers such that $n \ge 3$, $0 \le k \le n-1$, $0 < m \le n-1-k$. Then
\footnotesize
$$z(n,k,m) = \begin{cases}\displaystyle\sum_{f=1}^m {{k+f}\choose{f}}{{m-1}\choose{f-1}}z(n-m-f,k+f,0) &\mbox n+k+m\,{\rm odd}\\
\displaystyle\sum_{f=1}^m  {{k+f-1}\choose{f-1}}{{m-1}\choose{f-1}}z(n-m-f,k+f-1,0)+\\
\displaystyle\sum_{f=1}^m {{k+f}\choose{f}}{{m-1}\choose{f-1}}z(n-m-f,k+f,0) &\mbox n+k+m\,{\rm even}\end{cases}$$
\normalsize
\end{thm}

\begin{proof}
Define the mapping $\psi$ on a string $b \in Z(n,k,m)$ to be the string obtained by deleting from $b$ all substrings of consecutive 1s of length greater than 1; notice that this operation is well-defined, and that $\psi(b)$ is unique. Let $f$ be the number of substrings of 1s removed from $b$, so clearly $f$ is between 1 and $m$. The length of $\psi(b)$ must be $n-(m+f)$, since the removal of a substring of $g$ consecutive 1s removes only $g-1$ 1-pairs. Clearly $\psi(b)$ has no 1-pairs, since all such strings are removed, and no 1-pair can be created by our deletion process. Indeed, the removal of a string of consecutive 1s creates an additional 0-pair, unless that substring is removed from the end of the original string (it cannot come from the beginning since the string starts with a zero, by definition of $Z(n,k,m)$).

So to summarize, for any $b \in Z(n,k,m)$ ending in either 0 or 01, there exists a unique integer $1 \le f \le m$ and a unique binary string $\psi(b) \in Z(n-m-f,k+f,0)$ such that $b$ can be obtained from $\psi(b)$ by injecting a total of $m+f$ 1s into $f$ 0-pairs such that each injected substring contains at least two 1s. Also, for any $b \in Z(n,k,m)$ ending in 11, there exists a unique integer $1 \le f \le m$ and a unique binary string $\psi(b) \in Z(n-m-f,k+f-1,0)$ ending in 0 such that $b$ can be obtained from $\psi(b)$ by injecting a total of $m+f$ 1s into $f-1$ 0-pairs and at the end of $\psi(b)$ such that each injected substring contains at least two 1s.

To count the number of strings in $Z(n,k,m)$ that end in 0 or 01, we follow the recipe above. Given any $1 \le f \le m$, let $b' \in Z(n-m-f,k+f,0)$, for which there are $z(n-m-f,k+f,0)$ choices. We can pick any $f$ 0-pairs in $b'$, which can be done in ${k+f}\choose{f}$ ways. To inject our substrings of ones, since we know each substring must have at least two 1s, and the remaining $m-f$ 1s are ``identical balls" that need to be distributed into ``distinguishable urns", which can be done in ${m-1}\choose{f-1}$ ways. Each of the binary strings constructed this way is contained in $Z(n,k,m)$, and all are distinct as discussed above. Hence $Z(n,k,m)$ contains
$$\displaystyle\sum_{f=1}^m {{k+f}\choose{f}}{{m-1}\choose{f-1}}z(n-m-f,k+f,0)$$
strings ending in either 0 or 01.

To count the number of strings in $Z(n,k,m)$ that end in 11, we proceed as before. Given any $1 \le f \le m$, we pick $b' \in Z(n-m-f,k+f-1,0)$ ending in zero. From Lemma~\ref{end}, the number of possibilities for $b'$ is 0 if $n-m-f+k+f-1 = n-m+k-1$ is even, or equivalently if $n+k+m$ is odd. However, if $n+k+m$ is even, every string in $Z(n-m-f,k+f-1,0)$ ends in 0, so there are $z(n-m-f,k+f-1,0)$ choices for $b'$. Then from the $k+f-1$ 0-pairs in $b'$, we choose $f-1$ into which to inject 1s, which can be done in ${k+f-1}\choose{f-1}$ ways. Finally, we can distribute the $m+f$ 1s we need to add between these $f-1$ 0-pairs and at the end of $b'$, such that each injected string has at least two 1s, in ${m-1}\choose{f-1}$ ways, exactly as above. Therefore, if $n+k+m$ is odd, $Z(n,k,m)$ contains no strings ending in 11, while if $n+k+m$ is even, $Z(n,k,m)$ contains:
$$\displaystyle\sum_{f=1}^m  {{k+f-1}\choose{f-1}}{{m-1}\choose{f-1}}z(n-m-f,k+f-1,0)$$
strings ending in 11. This proves the result.
\end{proof}

\section{The case $m=0$ and Terquem's problem}

Theorem~\ref{generalm} reduces our problem of computing $z(n,k,m)$ to the values of $z(n,k,0)$, but does not help us find these values. Fortunately, when analyzing computational data for $z(n,k,m)$, we searched the OEIS~\cite{oeis} for the case $m=0$, which yielded an unexpected connection with the sequence A046854. This sequence, which we call $T(n,k)$, represents a triangle of numbers with $n \ge 0$, $0 \le k \le n-1$ defined via
$$T(n,k) = {{\left \lfloor \frac{n+k}{2} \right \rfloor}\choose {k}}$$
Combinatorially, this sequence arises as a solution to Terquem's problem~\cite{riordan}, namely providing the number of length $k$, increasing sequences of integers from $\{1 \ldots n\}$ that alternate parity and start with an odd number, as well as the number of length $k$, increasing sequences of integers from $\{1 \ldots n+1\}$ that alternate parity and start with an even number.

We can verify this numerical relationship with a quick induction proof:
\begin{thm}
\label{terquem}
Let $n,k$ be integers such that $n > 0$ and $0 \le k \le n-1$. Then $z(n,k,0) = T(n-1,k) = {{\left \lfloor \frac{n+k-1}{2} \right \rfloor}\choose {k}}$.
\end{thm}

\begin{proof}
We proceed via strong induction. It is easy to calculate $z(n,k,0)$ for small values of $n$ via enumeration:
\begin{description}
\item[$n=1$] $Z(1,0,0) = \{0\}$, $z(1,0,0) = T(0,0) = 1$
\item[$n=2$] $Z(2,0,0) = \{01\}$, $z(2,0,0) = 1$; $Z(2,1,0) = \{00\}$, $z(2,1,0)  = 1$
\item[$n=3$] $Z(3,0,0) = \{010\}$, $Z(3,1,0) = \{001\}$, $Z(3,2,0) = \{000\}$, \\$Z(3,0,1) =\{011\}$
\end{description}

By way of induction, assume $z(n,k,0) = {{\left \lfloor \frac{n+k-1}{2} \right \rfloor}\choose {k}}$ for all $n < N$, $0 \le k \le n-1$, and consider $z(N,k,0)$. By Theorem~\ref{recur1}, we have
\begin{eqnarray*}
z(N,k,0) &=& z(N-1,k-1,0) +z(N-2,k,0) + z(N-2,-1,k)\\
& = & z(N-1,k-1,0) +z(N-2,k,0) \\
\end{eqnarray*}
Using our induction hypothesis, we have $z(N,k,0) = {{\left \lfloor \frac{N+k-3}{2} \right \rfloor}\choose {k-1}} + {{\left \lfloor \frac{N+k-3}{2} \right \rfloor}\choose {k}}$. A simple application of Pascal's identity yields the result.
\end{proof}

Interestingly, it is possible to generate a bijection between all of the binary strings in $Z(n,k,0)$ and Terquem's sequences. A rigorous proof of this fact is tedious, but beginning with $b \in Z(n,k,0)$, define $t$ to be the sequence of positions of the first 0 in each 0-pair, where the first position in $b$ is position 1. The key is to note that to be in $Z(n,k,0)$, $b$ basically consists of runs of two or more 0s, with alternating strings of 0s and 1s between them; this alternation forces the first element of $t$ to be odd, as well as alternate parity thereafter. As an example:
$$001010001010001 \rightarrow 1,6,7,12,13$$

\bibliography{bibfile}{}
\bibliographystyle{plain}
\end{document}